\documentclass[12pt,a4paper]{article}
\usepackage[margin=2.54cm]{geometry}

\def\Dj{\hbox{D\kern-.73em\raise.30ex\hbox{-}
\raise-.30ex\hbox{}}}
\def\dj{\hbox{d\kern-.33em\raise.80ex\hbox{-}
\raise-.80ex\hbox{\kern-.40em}}}
\usepackage{graphicx}
\usepackage{amsmath,amsthm,amsfonts,amssymb,amscd,cite}
\allowdisplaybreaks

\newtheorem{thm}{Theorem}
\newtheorem{theo}{Theorem}
\newtheorem{lem}[theo]{Lemma}

\newtheorem{cor}[]{Corollary}

\newtheorem{prob}[]{Problem}

\newcommand{\beq}{\begin{eqnarray}}
\newcommand{\eeq}{\end{eqnarray}}

\newcommand{\beqs}{\begin{eqnarray*}}
\newcommand{\eeqs}{\end{eqnarray*}}

\begin{document}

\baselineskip=0.30in

\vspace*{25mm}

\begin{center}
{\LARGE \bf On the Extremal Zagreb Indices of $\mathbf{\textit{n}}$--Vertex \\[2mm] Chemical Trees with Fixed Number of \\[2mm]
Segments or Branching Vertices}

\vspace{5mm}

{\large \bf Sadia Noureen$^1$}, {\large \bf Akbar Ali}$^{2,3,}$\footnote{Corresponding author},
{\large \bf Akhlaq Ahmad Bhatti}$^1$

\vspace{3mm}

\baselineskip=0.20in

$^1${\it Department of Sciences and Humanities,\\ National University of Computer and Emerging Sciences, Lahore Campus,\\ B-Block, Faisal Town, Lahore, Pakistan} \\
{\tt sadia.tauseef@uog.edu.pk, akhlaq.ahmad@nu.edu.pk}\\[2mm]

$^2${\it Department of Mathematics, Faculty of Science,\\ University of Ha\!'il, Ha\!'il 81451, Saudi Arabia}\\
{\tt akbarali.maths@gmail.com}\\[2mm]

$^3${\it Knowledge Unit of Science, University of Management and Technology, \\
Sialkot 51310, Pakistan}

\vspace{6mm}


\end{center}

\begin{abstract}
Let $\mathcal{CT}_{n,k}$ and $\mathcal{CT}^*_{n,b}$ be the classes of all $n$-vertex chemical trees with $k$ segments and $b$ branching vertices, respectively, where $3\le k\le n-1$ and $1\le b< \frac{n}{2}-1$.
The solution of the problem of finding trees from the class $\mathcal{CT}_{n,k}$ or $\mathcal{CT}^*_{n,b}$, with the minimum first Zagreb index or minimum second Zagreb index follows directly from the main results of [MATCH Commun. Math. Comput. Chem. 72 (2014) 825--834] or [MATCH Commun. Math. Comput. Chem. 74 (2015) 57--79].
In this paper, the chemical trees with the maximum first/second Zagreb index  are characterized from each of the aforementioned graph classes.
\end{abstract}

\baselineskip=0.30in

\section{Introduction}
All the graphs discussed in this paper are simple and connected. Chemical compounds can be represented by graphs, known as chemical graphs, in which vertices correspond to atoms and edges represent the bonds of the considered chemical compound.  Let $G$ be a graph with vertex set $V(G)$ and edge set $E(G)$. If two vertices $u$ and $v$ of the graph $G$ are adjacent, then the edge connecting them will be denoted by $uv$. The number of vertices adjacent to the vertex $u\in V(G)$ is its degree, and it will be denoted by $d_u(G)$. In a chemical graph, every vertex has degree at most 4. Let $n_{i}(T)$ be the number of vertices of degree $i$ in a graph $G$.
Let $N_G(u)$ be the set of all those vertices of $G$ that are adjacent to the vertex $u\in V(G)$. A vertex of degree one is called a \emph{pendent} vertex. A vertex of degree more than two is known as a \emph{branching} vertex. A pendent vertex adjacent to a branching vertex is called a \emph{starlike pendent} vertex.
A graph with $n$ vertices is called $n$-vertex graph. When the graph under consideration is clear, we drop ``$G$'' from the graph theoretical notations -- for example, we write $d_u$, $n_i$ and $N(u)$ instead of $d_u(G)$, $n_i(G)$ and $N_G(u)$, respectively. If $V(G)=\{v_1,v_2,...,v_n\}$ then the sequence $(d_{v_1},d_{v_2},...,d_{v_n})$ is called the degree sequence of $G$ and it is usually assumed that $d_{v_1}\ge d_{v_2}\ge\cdots\ge d_{v_n}$. Undefined terminology and notations from (chemical) graph theory can be found in books \cite{Tri-1, Harary, Bondy}.

In chemical graph theory, the graph invariants (that found some chemical applications in chemistry) are called \emph{topological indices}.
Long time ago, a pair of topological indices were appeared within the study of the dependence of total $\pi$-electron energy of molecular structures \cite{Gut-72,Gut-75}. Nowadays, the members of this pair are known as the \emph{first Zagreb} index, which is denoted by $M_1$, and the \emph{second Zagreb} index, which is denoted by $M_2$. For a (molecular) graph $G$, these Zagreb indices are defined as
$$M_{1}(G)=\sum_{v\in V(G)}(d_{v})^{2} \ \ \  \text{and }\ \ \ M_{2}(G)=\sum_{uv\in E(G)}d_{u}d_{v}\,.$$
These indices were given different names in the literature, such as the Zagreb Group indices \cite{Gut-75}, the Zagreb group parameters \cite{Balban-83} and the Zagreb indices \cite{Todeschini-00}. The Zagreb indices attracted much interest from mathematical chemists and mathematicians, and as a result a plethora of their mathematical properties were reported -- detail about the mathematical theory and applications of these indices can be found in the recent surveys \cite{Bojana-17,Bojana17-2,Gut-18,Ali-M-18,Ali-MATCH-19}, recent papers \cite{Moj-19,Eliasi-18,Ashrafi-19,Pei-18,Das-DML-19,Ali-AEJM-18,Ali-Das-2019,Javaid-19,Milov-18,Yurtas-19,Zhan-19} and related references listed therein.

Let $P:u_0u_1u_2\cdots u_r$ be a path of length $r\ge2$ in a graph. The vertices $u_0$ and $u_r$ are called end vertices of $P$. If $r\ge3$ then the vertices $u_1,u_2,\cdots, u_{r-1}$ are called internal vertices of $P$. A \emph{pendent path} in a graph is a path in which one of the end vertices is pendent and the other is branching, and all the internal vertices (if exist) have degree 2. An \emph{internal path} in a graph is a path in which both the end vertices are branching and all the internal vertices (if exist) have degree 2. A \emph{segment} of a tree $T$ is a non-trivial path $P'$ in $T$ with the property that neither of the end vertices of $P'$ has degree 2 and that all the internal vertices (if exist) of $P'$ have degree 2.

Denote by $\mathcal{CT}_{n,k}$ and $\mathcal{CT}^*_{n,b}$ the classes of all $n$-vertex chemical trees with $k$ segments and $b$ branching vertices, respectively, where $1\le k\le n-1$ and $0\le b\le \frac{n}{2}-1$. The solution of the problem of finding trees from the class $\mathcal{CT}_{n,k}$ or $\mathcal{CT}^*_{n,b}$, with the minimal first Zagreb index or minimal second Zagreb index follows directly from the main results of \cite{boj-15} or \cite{lin-14}. The main purpose of the present paper is to solve the following chemical extremal graph theoretical problem.


\begin{prob}\label{prb-A0}
Characterize all the trees attaining the maximal first Zagreb index or maximal second Zagreb index from the
class $\mathcal{CT}_{n,k}$ or $\mathcal{CT}^*_{n,b}$.

\end{prob}

Clearly, the classes $\mathcal{CT}_{n,1}$ and $\mathcal{CT}^*_{n,0}$ consist of only the path graph and the class $\mathcal{CT}_{n,2}$ is empty. It is mentioned in the papers \cite{boj-15,lin-14} that the $n$-vertex star graph is the unique tree with $n-1$ segments -- however, this is not the case because
every $n$-vertex tree containing no vertex of degree 2 has $n-1$ segments. Also, if $T\in\mathcal{CT}^*_{n,\frac{n}{2}-1}$ then $T$ consists of the vertices only of degrees 1 and 3, and hence $M_1(T)=5n-8$, $M_2(T)=6n-15$, where $n\ge4$.
Thus, we solve Problem \ref{prb-A0} under the constraints $3\le k\le n-1$ and $1\le b< \frac{n}{2}-1$.
Moreover, if $k=3$, 4, the solution of the problem of characterizing trees from the class $\mathcal{CT}_{n,k}$ with the maximal first Zagreb index or maximal second Zagreb index follows directly from Theorem 1 of \cite{lin-14} or Theorem 3.1 of \cite{boj-15}, respectively. However, for the sake of completeness, we state our main results, concerning segments, with the condition $3\le k\le n-1$ instead of $5\le k\le n-1$.

\section{Statements of the Main Results}\label{sec-3}

This section is concerned with the statements of our main results, which give the solution of Problem \ref{prb-A0}. In order to state the first two of these results, we need the following elementary lemma.

\begin{lem}\label{zm}
For any tree $T \in \mathcal{CT}_{n,k}$, with $3\le k\le n-1$, the following results hold.\\
a) $n_3=0$ if and only if $k\equiv 1$ (mod 3), $n_1=\frac{2k+4}{3}$, $n_2=n-k-1$ and $n_4=\frac{k-1}{3}$.\\
b) $n_3=1$ if and only if  $k\equiv 0$ (mod 3), $n_1=\frac{2k+3}{3}$, $n_2=n-k-1$  and $n_4=\frac{k-3}{3}$.\\
c) $n_3=2$ if and only if  $k\equiv 2$ (mod 3), $n_1=\frac{2k+2}{3}$, $n_2=n-k-1$  and $n_4=\frac{k-5}{3}$.

\end{lem}

\begin{proof}
From the well known identities
\begin{equation}\label{eq-aa6}
n=n_1+n_2+n_3+n_4
\end{equation}
and
\begin{equation}\label{eq-aa7}
n_1+2n_2+3n_3+4n_4=2(n-1),
\end{equation}
it follows that
\begin{equation}\label{eq-aa4}
n_1=n_3+2n_4+2.
\end{equation}
By using \eqref{eq-aa4} in the equation $k=(n_1+n_3+n_4)-1$, we get
\begin{equation}\label{eq-aa5}
k\equiv 2n_3+1 \text{ \ (mod 3)}.
\end{equation}
Now, by using the identity $n_2=n-k-1$ (see \cite{lin-14} for details) in \eqref{eq-aa7}, we have
\begin{equation}\label{eq-aa8}
n_1+4n_4=2k-3n_3\,.
\end{equation}
By solving \eqref{eq-aa4} and \eqref{eq-aa8} for the unknowns $n_1$ and $n_4$, we get
\begin{equation}\label{eq-aa9}
n_1=\frac{2k-n_3+4}{3}
\end{equation}
and
\begin{equation}\label{eq-aa10}
n_4=\frac{k-2n_3-1}{3}\,.
\end{equation}
From
\eqref{eq-aa5},
\eqref{eq-aa9} and
\eqref{eq-aa10}, the desired results follow.

\end{proof}

Let $\mathcal{CT}_{0}(n,k)$, $\mathcal{CT}_{1}(n,k)$ and $\mathcal{CT}_{2}(n,k)$ be the subclasses of $\mathcal{CT}_{n,k}$ consisting of the trees that contain no vertex of degree 3, contain one vertex of degree 3 and contain two vertices of degree 3, respectively. Then, by Lemma \ref{zm}, every member of $\mathcal{CT}_{0}(n,k)$, $\mathcal{CT}_{1}(n,k)$ or $\mathcal{CT}_{2}(n,k)$ satisfies $k\equiv 1$ (mod 3), $k\equiv 0$ (mod 3) or $k\equiv 2$ (mod 3), respectively, and also that member has the degree sequence   $$(\underbrace{4,4,...,4}_{\frac{k-1}{3}},\underbrace{2,2,...,2}_{n-k-1},\underbrace{1,1,...,1}_{\frac{2k+4}{3}}),$$ $$(\underbrace{4,4,...,4}_{\frac{k-3}{3}},3,\underbrace{2,2,...,2}_{n-k-1},\underbrace{1,1,...,1}_{\frac{2k+3}{3}})$$ or $$(\underbrace{4,4,...,4}_{\frac{k-5}{3}},3,3,\underbrace{2,2,...,2}_{n-k-1},\underbrace{1,1,...,1}_{\frac{2k+2}{3}}),$$ respectively.

\begin{thm}\label{thm-1}
If $3\leq k \leq n-1$ and $CT\in \mathcal{CT}_{n,k}$, then
\[
M_{1}(CT)
\leq
\begin{cases}
4n+2k-10 &\text{if\  $k\equiv 0$ (mod 3), }\\
4n+2k-8 &\text{if\   $k\equiv 1$ (mod 3), }\\
4n+2k-12 &\text{if\  $k\equiv 2$ (mod 3).}\\
\end{cases}
\]
The equality holds if and only if  $CT\in \mathcal{CT}_{1}(n,k)$ for $k\equiv 0$ (mod 3), $CT \in \mathcal{CT}_{0}(n,k)$ for $k\equiv 1$ (mod 3), and $CT \in \mathcal{CT}_{2}(n,k)$ for $k\equiv 2$ (mod 3).

\end{thm}

Let $\mathcal{CT}_{0}'(n,k)$, $\mathcal{CT}_{1}'(n,k)$ and $\mathcal{CT}_{2}'(n,k)$ be the subclasses of $\mathcal{CT}_{0}(n,k)$, $\mathcal{CT}_{1}(n,k)$ and $\mathcal{CT}_{2}(n,k)$, respectively, consisting of the trees that satisfy the following properties:\\
$\bullet$ every internal path (if exists) has length $1$,\\
$\bullet$ if there is at least one starlike pendent vertex then there is no pendent path of length greater than $2$,\\
$\bullet$ every vertex of degree $3$ (if exists) does not have more than one branching neighbor,\\
$\bullet$ if there is a pendent neighbor of a vertex of degree $4$ then there is no vertex of degree $3$ having any neighbor of degree $2$,\\
$\bullet$ if $n_4>0$ then the graph induced by the vertices of degree 4 is a tree.

\begin{thm} \label{thm-2}
If $CT\in \mathcal{CT}_{n,k},$ with $3\leq k \leq n-1$, then it holds that
\[
M_{2}(CT)
\leq
\begin{cases}
6n+2k-24 &\text{\ if $n < \frac{5k}{3} + 1$ and \ $k\equiv 0$ (mod 3),}\\[2mm]
\frac{30n-14k-87}{3} &\text{\ if $n = \frac{5k}{3} + 1$ and \ $k\equiv 0$ (mod 3),}\\[2mm]
\frac{12n+16k-66}{3}&\text{\ if $n > \frac{5k}{3} + 1$ and \ $k\equiv 0$ (mod 3),}\\[2mm]
6n+2k-22&\text{\ if $n < \frac{5k+7}{3}$ and $k\equiv 1$ (mod 3), }\\[2mm]
\frac{12n+16k-52}{3}&\text{\ if $n \ge \frac{5k+7}{3}$ and $k\equiv 1$ (mod 3),}\\[2mm]
6n+2k-26&\text{\ if $n< \frac{5k-4}{3}$, \ $k\equiv 2$ (mod 3) and $k\ne 5$,}\\[2mm]
\frac{15n+11k-85}{3}&\text{\ if $\frac{5k-4}{3}\leq n \leq \frac{5k+2}{3}$, $k\equiv 2$ (mod 3) and $k\ne 5$,}\\[2mm]
\frac{12n+16k-80}{3}&\text{\ if $n > \frac{5k+2}{3}$,  $k\equiv 2$ (mod 3) and $k\ne 5$,}\\[2mm]
5n-9&\text{\ if $n < 10$ and $k = 5$,}\\[2mm]
4n+1&\text{\ if $n\ge 10$ and  $k = 5$.}
\end{cases}
\]
with equality if and only if  $CT\in \mathcal{CT^{\prime}}_{1}(n,k)$ for $k\equiv 0$ (mod 3), $CT \in \mathcal{CT^{\prime}}_{0}(n,k)$ for $k\equiv 1$ (mod 3), and $CT \in \mathcal{CT^{\prime}}_{2}(n,k)$ for $k\equiv 2$ (mod 3).

\end{thm}

Since $n_2=n-k-1$ (see \cite{lin-14} for details), we remark that
the solution of the problem of finding trees from the class of all $n$-vertex chemical trees having $n_2$ number of vertices of degree 2, with the maximal first Zagreb index or maximal second Zagreb index, follows from Theorem \ref{thm-1} or Theorem \ref{thm-2}, respectively, where $0\le n_2 \le n-4$.

For $1\leq b < \frac{n-2}{3}$ and for $\frac{n-2}{3}\le b< \frac{n}{2}-1$,
denote by $\mathcal{BT}_{1}(n,b)$ and by $\mathcal{BT}_{2}(n,b)$ the subclasses of $\mathcal{CT}^*_{n,b}$ consisting of the trees with the degree sequences
$$(\underbrace{4,4,...,4}_{b},\underbrace{2,2,...,2}_{n-3b-2},\underbrace{1,1,...,1}_{2b+2}\,)$$
and
$$(\underbrace{4,4,...,4}_{n-2b-2},\underbrace{3,3,...,3}_{3b-n+2},\underbrace{1,1,...,1}_{n-b}\,)$$
respectively.

\begin{thm}\label{thm-3}
If $BT\in \mathcal{CT}^*_{n,b}$ then
\[
M_{1}(BT)
\leq
\begin{cases}
2(2n+3b-3) &\text{if\  $1\leq b < \frac{n-2}{3}$, }\\[3mm]
2(4n-3b-7) &\text{if\  $\frac{n-2}{3}\le b < \frac{n}{2}-1$.}\\
\end{cases}
\]
The equality sign in the inequality $M_{1}(BT)\le 2(2n+3b-3)$ holds if and only if $BT \in \mathcal{BT}_{1}(n,b)$ for $1\leq b < \frac{n-2}{3}$ and the equality sign in the inequality $M_{1}(BT)\le 2(4n-3b-7)$ holds if and only if $BT\in \mathcal{BT}_{2}(n,b)$ for $\frac{n-2}{3}\le b< \frac{n}{2}-1$.

\end{thm}

For $1\leq b < \frac{n-2}{3}$ and for $\frac{n-2}{3}\le b< \frac{n}{2}-1$, denote by  $\mathcal{BT'}_{1}(n,b)$ and by $\mathcal{BT'}_{2}(n,b)$, the subclasses of $\mathcal{BT}_{1}(n,b)$ and $\mathcal{BT}_{2}(n,b)$, respectively, consisting of the trees that satisfy the following constraints:\\
$\bullet$ every internal path (if exists) has length 1,\\
$\bullet$ if there is a pendent vertex adjacent to a vertex of degree 4, then there is no adjacent vertices of degree 3,\\
$\bullet$ if there is a pendent vertex adjacent to a branching vertex, then there is no pendent path of length greater than 2,\\
$\bullet$ every vertex of degree 3 (if exists) has at most one neighbor of degree 4,\\
$\bullet$ $n_4>0$ and the graph induced by the vertices of degree 4 is a tree.

\begin{thm}\label{thm-4}
If $BT\in \mathcal{CT}^*_{n,b},$ where $1\leq b < \frac{n}{2}-1$, then
\[
M_{2}(BT)
\leq
\begin{cases}
4n+16b-12 &\text{if $1\leq b \le \frac{n-4}{5}$,}\\[2mm]
6n+6b-20 &\text{if $\frac{n-4}{5}< b < \frac{n-2}{3}$,}\\[2mm]
10n-6b-28 &\text{if $\frac{n-2}{3} \leq b < \frac{3n-4}{7}$,}\\[2mm]
16n-20b-36 &\text{if $\frac{3n-4}{7} \leq b < \frac{n}{2}-1$.}
\end{cases}
\]
The equality holds if and only if $BT \in \mathcal{BT'}_{1}(n,b)$ for $1\leq b < \frac{n-2}{3}$, and $BT\in \mathcal{BT'}_{2}(n,b)$ for $\frac{n-2}{3}\le b< \frac{n}{2}-1$.

\end{thm}


\section{Proofs of Theorems \ref{thm-1} and \ref{thm-2}}\label{sec-4}

Let $CT^{1}_{max}$ (respectively $CT^{2}_{max}$) be the tree with the maximal first Zagreb index (respectively, second Zagreb index) among all the members of the class $\mathcal{CT}_{n,k}$ where $3\le k\le n-1$. In order to prove Theorems \ref{thm-1} and \ref{thm-2}, we first establish some structural properties of the trees $CT^{1}_{max}$ and $CT^{2}_{max}$.

\begin{lem}\label{sm3b}
The tree $CT^{1}_{max}\in \mathcal{CT}_{n,k}$ (respectively $CT^{2}_{max}\in \mathcal{CT}_{n,k})$ contains at most two vertices of degree 3 where $3\le k\le n-1$.

\end{lem}

\begin{proof}
We give a proof by contradiction. Suppose that the tree $CT^{1}_{max}$ (respectively $CT^{2}_{max}$) contains the vertices $u$, $v$ and $w$ of degree 3. We may assume that the vertex $v$ lies on the $u$-$w$ path. Let $w_1,w_2$ be the neighbors of $w$ that do not lie on the $u$-$w$ path. Let $T^{\prime}$ be the tree obtained from $CT^{1}_{max}$ (respectively $CT^{2}_{max}$) by deleting the edges $ww_1$,$ww_2$ and adding the edges $uw_1$,$vw_2$, then it is clear that $T^{\prime}\in \mathcal{CT}_{n,k}$\,.
Denote by $d_x$ the degree of a vertex $x$ in $CT^{1}_{max}$ (respectively in $CT^{2}_{max}$). It can be easily checked that
\[
M_{1}(CT^{1}_{max})-M_{1}(T^{\prime})<0\,,
\]
which is a contradiction to the definition of $CT^{1}_{max}$.\\
Next, we show that $M_{2}(CT^{2}_{max})-M_{2}(T^{\prime})<0,$ which contradicts the definition of $CT^{2}_{max}$.
Let $w_3$ be the unique neighbor of $w$ that lies on the path $u$-$w$. By definition of $M_2$, it holds that
\beq\label{eq-a1}
M_{2}(CT^{2}_{max})-M_{2}(T^{\prime})&=& \sum_{x\in N_G(u)}3d_{x}+\sum_{y\in N_G(v)}3d_{y}+\sum_{i=1}^{3}3d_{w_i}\nonumber\\
&&-\sum_{x\in N_G(u)}4d_{x}-\sum_{y\in N_G(v)}4d_{y}-4d_{w_{1}}-4d_{w_{2}}-d_{w_3}\nonumber\\
&=&2d_{w_3}-d_{w_{1}}-d_{w_{2}} -\sum_{x\in N_G(u)}d_{x} - \sum_{y\in N_G(v)}d_{y}.
\eeq
The right hand side of \eqref{eq-a1} is negative due to the facts that $\sum_{x\in N(u)}d_{x}\ge 4$,  $\sum_{y\in N(v)}d_{y}\ge 5$ and $d_{w_3}\le 4$.
This completes the proof.

\end{proof}

We can now prove Theorem \ref{thm-1}.

\begin{proof}[\bf Proof of Theorem \ref{thm-1}]
Recall that we have denoted by $CT^{1}_{max}$ the tree attaining the maximal first Zagreb index among all the members of $\mathcal{CT}_{n,k}$. By Lemma \ref{sm3b}, $CT^{1}_{max}$ must have at most two vertices of degree 3 and hence by Lemma \ref{zm}, we have
\[
M_{1}(CT^{1}_{max})
=
\begin{cases}
4n+2k-10 &\text{if\  $k\equiv 0$ (mod 3), }\\
4n+2k-8 &\text{if\   $k\equiv 1$ (mod 3), }\\
4n+2k-12 &\text{if\  $k\equiv 2$ (mod 3).}\\
\end{cases}
\]
Now, bearing in mind the definitions (see Section \ref{sec-3}) of $\mathcal{CT}_{0}(n,k)$, $\mathcal{CT}_{1}(n,k)$ and $\mathcal{CT}_{2}(n,k)$, we get the desired result.
\end{proof}

In order to prove Theorem \ref{thm-2}, we need to establish some further structural properties of the tree $CT^{2}_{max}$.

\begin{lem}\label{sm1}
For $3\le k\le n-1$, the tree $CT^{2}_{max}\in \mathcal{CT}_{n,k}$ does not contain any internal path of length greater than 1.
\end{lem}

\begin{proof}
Assume, on the contrary, that there is an internal path $v_{0}v_{1}\cdots v_{r-1}v_{r}$ of length at least 2 in $CT^{2}_{max}$ where $v_0$ and $v_{r}$ are branching vertices and $d_{v_{1}}=d_{v_{2}}=\dots=d_{v_{r-1}}=2$. Let $u$ be a pendent vertex adjacent to some vertex $v\in V(CT^{2}_{max})$. The vertex $v$ may or may not be coincident with either of the vertices $v_0$ and $v_r$. If $CT^{\prime}$ is the tree obtained from $CT^{2}_{max}$ as follows:\\
$CT^{\prime}=CT^{2}_{max}-\{uv,v_0v_1,v_{r-1}v_r\}+\{v_0v_r,uv_1,v_{r-1}v\}$,\\
then $CT^{\prime}\in \mathcal{CT}_{n,k}$.
Whether the vertex $v$ is coincident with either of the vertices $v_0$ and $v_r$, or not, in both cases we have
\beq\label{aa-11}
M_{2}(CT^{2}_{max})-M_{2}(CT^{\prime})
&=& 2d_{v_{0}}+2d_{v_{r}}-d_{v_{0}}d_{v_{r}}-d_{v}-2\nonumber\\
&\leq& -4+2(d_{v_{0}}+d_{v_{r}})-d_{v_{0}}d_{v_{r}}.
\eeq
The right hand side of \eqref{aa-11} is negative because the function $f$ defined by $f(x,y)=2(x+y)-xy-4$, with $3\le x,y\le4$, is decreasing in both $x$ and $y$, and hence we have
$M_{2}(CT^{2}_{max})<M_{2}(CT^{\prime}),$ which is a contradiction to the choice of $CT^{2}_{max}$.
\end{proof}

\begin{lem}\label{sm2}
If the tree $CT^{2}_{max}\in \mathcal{CT}_{n,k}$ contains a pendent vertex adjacent to a branching vertex, then $CT^{2}_{max}$ does not contain a pendent path of length greater than 2 where $3\le k\le n-1$.
\end{lem}

\begin{proof}
Suppose, on the contrary, that $v_1v_2\cdots v_r$ is a pendent path of length at least 3 and there is a pendent vertex $u\in V(CT^{2}_{max})$ adjacent to some branching vertex $v\in V(CT^{2}_{max})$, where $v_1$ is a pendent vertex and $v_r$ is a branching vertex (the vertex $v_r$ may coincides with the vertex $v$). Let $CT^{\prime}=CT^{2}_{max}-\{uv,v_1v_2,v_2v_3\} + \{uv_2,v_2v,v_1v_3\}$.
Certainly, the tree $CT^{\prime}$ belongs to the class $\mathcal{CT}_{n,k}$ and from the fact $d_{v}\geq 3$, it follows that $M_{2}(CT^{2}_{max})-M_{2}(CT^{\prime})=-d_{v}+2 <0$, which is a contradiction to the choice of $CT^{2}_{max}$.
\end{proof}

\begin{lem}\label{sm3}
If the tree $CT^{2}_{max}\in \mathcal{CT}_{n,k}$ contains a pendent vertex adjacent to a vertex of degree 4 then $CT^{2}_{max}$ does not contain any vertex of degree 3 adjacent to a vertex of degree 2 where $3\le k\le n-1$.

\end{lem}

\begin{proof}
Suppose, on the contrary, that $v\in V(CT^{2}_{max})$ is a vertex of degree 3 adjacent to a vertex $u$ of degree 2 and $p\in V(CT^{2}_{max})$ is a pendent vertex adjacent to some vertex $w$ of degree 4. Let $t$ be the neighbor of $u$ different from $v$. Because of Lemma \ref{sm1}, $t$ must be different from $w$. If $CT^{\prime}=CT^{2}_{max}-\{tu,uv,pw\}+\{tv,pu,uw\}$ then we have $M_{2}(CT^{2}_{max})-M_{2}(CT^{\prime}) = -d_t<0$, which is a contradiction to the definition of $CT^{2}_{max}$.
\end{proof}

\begin{lem}\label{sm3-new}
If the tree $CT^{2}_{max}\in \mathcal{CT}_{n,k}$ contains a vertex $u$ of degree 3 then $u$ does not have more than one branching neighbor where $3\le k\le n-1$.

\end{lem}

\begin{proof}
Suppose, on the contrary, that $v$ and $w$ are two branching neighbors of $u$. Let $P=v_1v_2\cdots v_{i-1}v_iv_{i+1}\cdots v_r$ be the longest path containing $u$, $v$ and $w$, where $v_{i-1}=v$, $v_{i}=u$ and $v_{i+1}=w$.
By Lemma \ref{sm3b}, $P$ contains at most two vertices of degree 3 including $u$. If $P$ has two vertices of degree 3 including $u$ then, without loss of generality, we assume that $d_{v_j}=3$ for some $j$, where $1\le j \le i-1$. Thus, there exists some $k$ with $i+1\le k \le r-1$ such that $v_k$ has exactly one branching neighbor and $d_{v_k}=4$. If $CT^{\prime}=CT^{2}_{max}-\{v_{i-1}v_i,v_{i}v_{i+1},v_{k}v_{k+1}\} + \{v_{i-1}v_{i+1},v_{k}v_{i},v_{i}v_{k+1}\}$ then bearing in mind the facts $d_{v_{k+1}}\le 2$, $d_{v_{i+1}}=4$ and $d_{v_{i-1}}=3$ or 4, we have
\begin{align*}
M_{2}(CT^{2}_{max})-M_{2}(CT^{\prime})=&  -d_{v_{i-1}} + d_{v_{k+1}} < 0\,,
\end{align*}
a contradiction to the definition of $CT^{2}_{max}$.
\end{proof}

The next corollary follows directly from Lemmas
\ref{sm1}
and
\ref{sm3-new}.

\begin{cor}\label{sm3-cor}
If the maximum degree of the tree $CT^{2}_{max}\in \mathcal{CT}_{n,k}$ is 4 then the graph induced by the vertices of degree 4 of $CT^{2}_{max}$ is a tree where $3\le k\le n-1$.

\end{cor}

Denote by $x_{i,j}(G)$ (or simply by $x_{i,j}$) the number of edges in a graph $G$ connecting the vertices of degrees $i$ and $j$.
The following system of equations holds for any chemical tree $T$:
\begin{equation}\label{Eq557b}
\sum_{ \substack{ 1\leq i\leq 4 \\
         i\neq j}}x_{j,i}+2x_{j,j}=j\cdot n_{j}
\end{equation}
where $j=1,2,3,4$.

We are now in position to prove Theorem \ref{thm-2}.

\begin{proof}[\bf Proof of Theorem \ref{thm-2}]
Recall that we have denoted by $CT^{2}_{max}$ the tree attaining the maximal second Zagreb index among all the members of $\mathcal{CT}_{n,k}$. Thus, $M_2(CT)\le M_2(CT^{2}_{max})$ with equality if and only if $CT \cong CT^{2}_{max}$. If $k=3$, 4, the desired result follows from Theorem 3.1 of \cite{boj-15}.
In what follows, we determine $M_2(CT^{2}_{max})$ under the assumption $5\le k\le n-1$.

By Lemma \ref{sm3b}, the tree $CT^{2}_{max}$ contains at most two vertices of degree 3 and hence by Lemma \ref{zm}, the degree sequence $DS(CT^{2}_{max})$ of $CT^{2}_{max}$ is
\[
DS(CT^{2}_{max})
=
\begin{cases}
(\underbrace{4,4,...,4}_{\frac{k-3}{3}},3,\underbrace{2,2,...,2}_{n-k-1},\underbrace{1,1,...,1}_{\frac{2k+3}{3}}) &\text{if\  $k\equiv 0$ (mod 3), }\\
(\underbrace{4,4,...,4}_{\frac{k-1}{3}},\underbrace{2,2,...,2}_{n-k-1},\underbrace{1,1,...,1}_{\frac{2k+4}{3}}) &\text{if\   $k\equiv 1$ (mod 3), }\\
(\underbrace{4,4,...,4}_{\frac{k-5}{3}},3,3,\underbrace{2,2,...,2}_{n-k-1},\underbrace{1,1,...,1}_{\frac{2k+2}{3}}) &\text{if\  $k\equiv 2$ (mod 3).}\\
\end{cases}
\]
Thus, by Lemmas
\ref{sm1}--\ref{sm3-new} and Corollary \ref{sm3-cor} one can conclude that the tree $CT^{2}_{max}$ belongs to $\mathcal{CT}_{0}'(n,k)$, $\mathcal{CT}_{1}'(n,k)$ or $\mathcal{CT}_{2}'(n,k)$.\\

\textit{Case 1.} The tree $CT^{2}_{max}$ is a member of $\mathcal{CT}_{0}'(n,k)$.\\
We note that $CT^{2}_{max}$ has the degree sequence $$(\underbrace{4,4,...,4}_{\frac{k-1}{3}},\underbrace{2,2,...,2}_{n-k-1},\underbrace{1,1,...,1}_{\frac{2k+4}{3}})$$ and  the congruence $k\equiv$ 1 (mod 3) holds. Because of the assumption $k\ge5$, we have $n_4\ge1$. By Corollary \ref{sm3-cor}, it holds that
\begin{equation}\label{Eq557b-a1}
x_{4,4}=n_4-1=\frac{k-4}{3}.
\end{equation}

\textit{Subcase 1.1.} The inequality $n < \frac{5k+7}{3}$ holds.\\
From the inequality $n < \frac{5k+7}{3}$, we have $n_1>n_2$ and thus (by Lemmas \ref{sm1} and \ref{sm2}), it holds that
\begin{equation}\label{Eq557b-a2}
x_{2,2}=0.
\end{equation}
From
\eqref{Eq557b}, \eqref{Eq557b-a1}
 and \eqref{Eq557b-a2}, it follows that $x_{{1,2}}=x_{2,4}=n-k-1$, $x_{1,4}=\frac{5k-3n+7}{3}$.\\
 Hence
 $$M_2(CT^{2}_{max})=6n+2k-22.$$

\textit{Subcase 1.2.} $n \ge \frac{5k+7}{3}$.\\
In this subcase, it holds that $n_1\le n_2$
and hence (by using Lemmas \ref{sm1} and \ref{sm2}) we have
\begin{equation}\label{Eq557b-a33}
x_{1,4}=0.
\end{equation}
From
\eqref{Eq557b}, \eqref{Eq557b-a1}
and \eqref{Eq557b-a33}, it follows that
$x_{{1,2}}=x_{2,4}=\frac{2k+4}{3}$, $x_{2,2}=\frac{3n-5k-7}{3}$, therefore we have
$$M_2(CT^{2}_{max})=\frac{12n+16k-52}{3}.$$

\textit{Case 2.} $CT^{2}_{max}\in\mathcal{CT}_{1}'(n,k)$.\\
In this case, the tree $CT^{2}_{max}$ has the degree sequence $$(\underbrace{4,4,...,4}_{\frac{k-3}{3}},3,\underbrace{2,2,...,2}_{n-k-1},\underbrace{1,1,...,1}_{\frac{2k+3}{3}})$$ and  the congruence $k\equiv$ 0 (mod 3) holds, which implies that $k\ge 6$ (because of the assumption $k\ge5$). Thus, $n_4\ge1$ and hence by Corollary \ref{sm3-cor}, it holds that
\begin{equation}\label{Eq557b-a1-new}
x_{4,4}=n_4-1=\frac{k-6}{3}.
\end{equation}
Also, it holds that
\begin{equation}\label{Eq557b-a33a0}
x_{3,3}=0.
\end{equation}
By Lemmas
\ref{sm1}
and
\ref{sm3-new}, we have
\begin{equation}\label{Eq557b-a33a}
x_{3,4}=1.
\end{equation}
We note that $x_{2,2}=0$ and $x_{1,4}\ne0$ if $n_2<2n_4+2$; $x_{2,2}=x_{1,4}=0$ if $n_2=2n_4+2$; $x_{1,4}=0$ and $x_{2,2}\ne0$ if $n_2>2n_4+2$. We discuss these three cases in the following.\\

\textit{Subcase 2.1.} $n< \frac{5k}{3}+1$.\\
The inequality $n< \frac{5k}{3}+1$ implies that $n_2<2n_4+2$ and hence, it holds that
\begin{equation}\label{Eq557b-a34}
x_{2,2}=0
\end{equation}
and $x_{1,4}\ne0$, and hence (by Lemma \ref{sm3})
\begin{equation}\label{Eq557b-a35}
x_{2,3}=0.
\end{equation}
From \eqref{Eq557b}, \eqref{Eq557b-a1-new}, \eqref{Eq557b-a33a0},
\eqref{Eq557b-a33a}, \eqref{Eq557b-a34} and \eqref{Eq557b-a35}, it follows that
$x_{2,4}=x_{1,2}=n-k-1$, $x_{1,3}=2$, $x_{1,4}=\frac{5k-3n}{3}$ and hence
$$M_2(CT^{2}_{max})=6n+2k-24.$$

\textit{Subcase 2.2.} $n= \frac{5k}{3}+1$.\\
From $n= \frac{5k}{3}+1$, it follows that $n_2=2n_4+2$ and hence we have
\begin{equation}\label{Eq557b-a36}
x_{2,2}=x_{1,4}=0.
\end{equation}
From \eqref{Eq557b}, \eqref{Eq557b-a1-new}, \eqref{Eq557b-a33a0},
\eqref{Eq557b-a33a} and \eqref{Eq557b-a36}, it follows that
$x_{1,2}=n-k-1$, $x_{1,3}=1$, $x_{2,3}=1$, $x_{2,4}=n-k-2$, and hence
$$M_2(CT^{2}_{max})=\frac{30n-14k-87}{3}.$$

\textit{Subcase 2.3.} $n> \frac{5k}{3}+1$.\\
The inequality $n> \frac{5k}{3}+1$ yields $n_2>2n_4+2$, which further implies that
\begin{equation}\label{Eq557b-a37}
x_{1,4}=0.
\end{equation}
and $x_{2,2}\ne0$, and hence (by Lemmas \ref{sm1} and \ref{sm2})
\begin{equation}\label{Eq557b-a38}
x_{1,3}=0.
\end{equation}
From \eqref{Eq557b}, \eqref{Eq557b-a1-new}, \eqref{Eq557b-a33a0},
\eqref{Eq557b-a33a}, \eqref{Eq557b-a37} and \eqref{Eq557b-a38}, it follows that
$x_{1,2}=\frac{2k+3}{3}$, $x_{2,2}=\frac{3n-5k-6}{3}$, $x_{2,3}=2$, $x_{2,4}=\frac{2k-3}{3}$ and hence
$$M_2(CT^{2}_{max})=\frac{12n+16k-66}{3}.$$

\textit{Case 3.} $CT^{2}_{max}\in\mathcal{CT}_{2}'(n,k)$.\\
In this case, the tree $CT^{2}_{max}$ has the degree sequence $$(\underbrace{4,4,...,4}_{\frac{k-5}{3}},3,3,\underbrace{2,2,...,2}_{n-k-1},\underbrace{1,1,...,1}_{\frac{2k+2}{3}})$$ and  the congruence $k\equiv$ 2 (mod 3) holds. If $k=5$ then $n_4=0$, $x_{3,3}=1$ and hence
$$M_2(CT^{2}_{max})=
\begin{cases}
5n-9 & \text{ \ if $6\le n\le 10$},\\[2mm]
4n+1 & \text{ \ if $n> 10$}.
\end{cases}
$$
Next, in what follows, we assume $k\ge8$, which implies that $n_4\ge1$.
By Corollary \ref{sm3-cor}, it holds that
\begin{equation}\label{Eq557b-a1-new1}
x_{4,4}=n_4-1=\frac{k-8}{3}.
\end{equation}
By Lemmas
\ref{sm1}
and
\ref{sm3-new}, we have
\begin{equation}\label{Eq557b-a33ab}
x_{3,3}=0 \quad \text{and} \quad x_{3,4}=2.
\end{equation}

\textit{Subcase 3.1.} $n\le \frac{5k-7}{3}$.\\
The inequality $n\le \frac{5k-7}{3}$ implies that $n_2\le2n_4$ and hence, it holds that
\begin{equation}\label{Eq557b-a34b}
x_{2,2}=0
\end{equation}
and $x_{1,4}\ne0$, and hence (by Lemma \ref{sm3})
\begin{equation}\label{Eq557b-a35b}
x_{2,3}=0.
\end{equation}
From \eqref{Eq557b}, \eqref{Eq557b-a1-new1},
\eqref{Eq557b-a33ab}, \eqref{Eq557b-a34b} and \eqref{Eq557b-a35b}, it follows that
$x_{1,2}=x_{2,4}=n-k-1$, $x_{1,3}=4$, $x_{1,4}=\frac{5k-3n-7}{3}$ and hence
$$M_2(CT^{2}_{max})=6n+2k-26.$$

\textit{Subcase 3.2.} $\frac{5k-4}{3}\le n\le \frac{5k+2}{3}$.\\
From  $\frac{5k-4}{3}\le n\le \frac{5k+2}{3}$, it follows that $2n_4+1\le n_2\le 2n_4+3$ and hence we have
\begin{equation}\label{Eq557b-a36b}
x_{2,2}=x_{1,4}=0.
\end{equation}
From \eqref{Eq557b}, \eqref{Eq557b-a1-new1},
\eqref{Eq557b-a33ab} and \eqref{Eq557b-a36b}, it follows that
$x_{1,2}=n-k-1$, $x_{2,3}=\frac{3n-5k+7}{3}$, $x_{2,4}=2(\frac{k-5}{3})$, $x_{1,3}=\frac{5k-3n+5}{3}$ and hence
$$M_2(CT^{2}_{max})=\frac{15n+11k-85}{3}.$$

\textit{Subcase 3.3.} $n> \frac{5k+2}{3}$.\\
The inequality $n> \frac{5k+2}{3}$ yields $n_2>2n_4+3$, which further implies that
\begin{equation}\label{Eq557b-a37b}
x_{1,4}=0.
\end{equation}
and $x_{2,2}\ne0$, and hence (by Lemmas \ref{sm1} and \ref{sm2})
\begin{equation}\label{Eq557b-a38b}
x_{1,3}=0.
\end{equation}
From \eqref{Eq557b}, \eqref{Eq557b-a1-new1},
\eqref{Eq557b-a33ab}, \eqref{Eq557b-a37b} and \eqref{Eq557b-a38b}, it follows that
$x_{1,2}=\frac{2k+2}{3}$, $x_{2,2}=\frac{3n-5k-5}{3}$, $x_{2,3}= 4$, $x_{2,4}=2(\frac{k-5}{3})$ and hence
$$M_2(CT^{2}_{max})=\frac{12n+16k-80}{3}.$$
This completes the proof.
\end{proof}

\section{Proofs of Theorems \ref{thm-3} and \ref{thm-4}}\label{sec-5}

Let $C^{\prime}T^{1}_{max}$ (respectively $C^{\prime}T^{2}_{max}$) be the tree with the maximal $M_{1}$ (respectively, $M_2$) value among all members of $\mathcal{CT}^*_{n,b}$ for $1\leq b < \frac{n}{2}-1)$. We need to prove some lemmas first, to prove Theorems \ref{thm-3} and \ref{thm-4}.

\begin{lem}\label{lem-bm4}
Let $1\leq b < \frac{n}{2}-1$. If the tree $C^{\prime}T^{1}_{max}\in \mathcal{CT}^*_{n,b}$ (respectively $C^{ \prime} T^{2}_{max}\in \mathcal{CT}^*_{n,b}$) contains some vertex/vertices of degree 2, then it does not contain any vertex of degree 3. That is, the tree $C^{\prime}T^{1}_{max}\in \mathcal{CT}^*_{n,b}$ (respectively $C^{ \prime} T^{2}_{max}\in \mathcal{CT}^*_{n,b}$) does not contain the vertices of degrees 2 and 3 simultaneously.
\end{lem}

\begin{proof}
On the contrary, we assume that the conclusion of the lemma is wrong and that the hypothesis of the lemma is true. Let $z$ be a vertex of degree 3 in $C^ {\prime}T^{1}_{max}$ (respectively $C^{\prime}T^{2}_{max}$). We take a vertex $v$ of degree 2 with neighbors $u$ and $w$ such that $d_u\ge 1$ and $d_w\ge3$. Let $N(z)=\{z_1,z_2,z_3\}$ where the vertices $z_1$ and $z_2$ do not lie on the unique $v-z$ path (it is possible that the vertex $z$ or $z_3$ is coincident with $u$ or $w$, and if $z=u$ or $w$ then $z_3=v$). If $T^{\prime}$ is the tree obtained from $C^{\prime}T^{1}_{max}$ (respectively $C^{\prime}T^{2}_{max}$) by deleting the edges $z_1z$, $z_2z$ and adding the edges $vz_1$, $vz_2$, then it can be observed that $T^{\prime}\in \mathcal{CT}^*_{n,b}$, and that
\beqs
M_{1}(C^{\prime}T^{1}_{max})- M_{1}(T^{\prime})
&=& -4 < 0
\eeqs
which is a contradiction to the choice of $C^{\prime}T^{1}_{max}$.

\noindent Also, keeping in mind the facts $d_u\ge1$, $d_{z_1}\ge 1$, $d_{z_2}\ge 1$, $d_w\ge3$ and $d_{z_3}\le 4$, we get
\beqs
M_{2}(C^ {\prime}T^{2}_{max})- M_{2}(T^{\prime})
&=& 2d_{z_3}-2d_u-2d_w-d_{z_1}-d_{z_2}\\
&\le& 4-2d_w < 0,
\eeqs
which is again a contradiction to the definition of $C^{\prime}T^{2}_{max}$.
\end{proof}

\begin{lem}\label{zb-A1}
Let $1\leq b < \frac{n}{2}-1$. For the tree $C^{\prime}T^{1}_{max}\in \mathcal{CT}^*_{n,b}$ (respectively $C^{\prime}T^{2}_{max}\in \mathcal{CT}^*_{n,b}$), the following statements hold:\\
a)  if $n_2 >0$ then  $n_1=2b+2$, $n_2=n-3b-2$, $n_3=0$ and $n_4=b$;\\
b)  $n_2=0$ if and only if $n_1=n-b$, $n_3=3b-n+2$ and $n_4=n-2b-2$.
\end{lem}

\begin{proof}
a) We note that
\begin{align}\label{Eqb-Aa1}
n_3+n_4=b.
\end{align}
Since $n_2>0$, by Lemma \ref{lem-bm4}, it holds that
\begin{align}\label{Eqb-Aa2}
n_3=0.
\end{align}
From Equations \eqref{eq-aa6}, \eqref{eq-aa7}, \eqref{Eqb-Aa1} and \eqref{Eqb-Aa2}, it follows that
$n_1=2b+2$, $n_2=n-3b-2$ and $n_4=b$.\\

\noindent b) If $n_1=n-b$, $n_3=3b-n+2$ and $n_4=n-2b-2$ then Equation \eqref{eq-aa6} yields $n_2=0$. Conversely, suppose that $n_2=0$. Bearing in mind the assumption $n_2=0$ and by solving Equations \eqref{eq-aa6}, \eqref{eq-aa7}, \eqref{Eqb-Aa1}, we get
$n_1=n-b$, $n_3=3b-n+2$ and $n_4=n-2b-2$.
\end{proof}

\begin{lem}\label{zb-A2}
For the tree $C^{\prime}T^{1}_{max}\in \mathcal{CT}^*_{n,b}$ (respectively $C^{\prime}T^{2}_{max}\in \mathcal{CT}^*_{n,b}$), the inequality
$n_2 >0$ holds if and only if  $1\le b<\frac{n-2}{3}$ where $1\leq b < \frac{n}{2}-1$.
\end{lem}

\begin{proof}
If $n_2>0$, then by using Lemma \ref{zb-A1}$(a)$ we have $n_2=n-3b-2$ and hence $b<\frac{n-2}{3}$. Conversely, suppose that $1\le b<\frac{n-2}{3}$, that is $n\ge 3b+3$ with $b\ge1$. We have to show that $n_2>0$ and we will prove it by induction on $b$.
For $b=1$, we have $n\ge6$ and the graph in this case is the starlike tree with maximum degree at most 4, and hence the result is true for $b=1$. Assume that every chemical tree of order at least $3k+3$ with exactly $k$ branching vertices contains at least one vertex of degree 2, where $k\ge1$. Let $C^{\prime}T^{1}_{max}$ (respectively $C^{\prime}T^{2}_{max}$) be the chemical tree of order $n\ge 3(k+1)+3$ with exactly $k+1$ branching vertices. We have to show that $n_2>0$. Contrarily, suppose that $n_2=0$. By Lemma \ref{zb-A1}$(b)$, $n_4=n-2(k+1)-2>0$ because $n\ge 3(k+1)+3$.

We claim that $x_{1,4}\ne 0$. If $x_{1,4}= 0$ then the identity $x_{1,3}+ x_{1,4}=n_1$ gives $x_{1,3}=n_1$ and hence any branching vertex has at most two pendant neighbors, and thus it holds that $n_1\le 2(k+1)$. Also, the inequality $n\ge 3(k+1)+3$ implies that $n_1=n-(k+1)\ge 2(k+1)+3$ (because of Lemma \ref{zb-A1}$(b)$), which is a contradiction to the inequality $n_1\le 2(k+1)$. Thus, $x_{1,4}\ne 0$.

Now, let $P:u_1u_2\cdots u_{r-1}u_r$ be the longest path in $C^{\prime}T^{1}_{max}$ (respectively $C^{\prime}T^{2}_{max}$). We note that $u_2$ and $u_{r-1}$ are the branching vertices and that every neighbor, not lying on the path $P$, of either of these two vertices is pendent. If $d_{u_2}=4$ then let $T'$ be the graph obtained from $C^{\prime}T^{1}_{max}$ (respectively $C^{\prime}T^{2}_{max}$) by removing all the pendent neighbors of $u_2$ and if $d_{u_2}=3$ then let $T'$ be the graph obtained from $C^{\prime}T^{1}_{max}$ (respectively $C^{\prime}T^{2}_{max}$) by removing all the pendent neighbors of $u_2$ and removing a pendent neighbor of a vertex of degree 4. Clearly, the tree has order at least $3k+3$ and exactly $k$ branching vertices. Hence, by induction hypothesis $T'$ contains at least one vertex of degree 2. Thus, the tree $C^{\prime}T^{1}_{max}$ (respectively $C^{\prime}T^{2}_{max}$) has also at least one vertex of degree 2. This completes the induction and hence the proof.
\end{proof}

\begin{proof}[\bf Proof of Theorem \ref{thm-3}]
Recall that we have denoted by $C^{\prime}T^{1}_{max}$ the tree attaining the maximal first Zagreb index among all the members of $\mathcal{CT}^*_{n,b}$. By Lemma \ref{lem-bm4}, $C^{\prime}T^{1}_{max}$ cannot contain the vertices of degrees 2 and 3 simultaneously and hence by Lemmas \ref{zb-A1} and \ref{zb-A2}, we have
\[
M_{1}(C^{\prime}T^{1}_{max})
=
\begin{cases}
4n+6b-6 &\text{if\  $1\leq b < \frac{n-2}{3}$, }\\[3mm]
8n-6b-14 &\text{if\  $\frac{n-2}{3}\le b \le \frac{n}{2}-1$.}\\
\end{cases}
\]
Now, bearing in mind the definitions of $\mathcal{BT}_{1}(n,b)$ and $\mathcal{BT}_{2}(n,b)$ (see Section \ref{sec-3}), we get the desired result.
\end{proof}

In what follows, we prove some further structural properties of the tree $C^{\prime}T^{2}_{max}$, which are needed to prove Theorem \ref{thm-4}.

\begin{lem}\label{lem-bm1}
For $1\leq b < \frac{n}{2}-1$, the tree $C^{\prime}T^{2}_{max}\in \mathcal{CT}^*_{n,b}$ does not contain any internal path of length greater than 1.
\end{lem}

\begin{proof}
The proof is fully analogous to that of Lemma \ref{sm1}.
\end{proof}

\begin{lem}\label{lem-bm2}
If the tree $C^{\prime}T^{2}_{max}\in \mathcal{CT}^*_{n,b}$ contains a pendent vertex adjacent to a vertex of degree 4, then $C^{\prime}T^{2}_{max}$ does not contain adjacent vertices of degree 3 where $1\leq b < \frac{n}{2}-1$.
\end{lem}

\begin{proof}
Assume, on the contrary, that $w,z\in V(C^{\prime}T^{2}_{max})$ are the adjacent vertices of degree 3 and that $u\in V(C^{\prime}T^{2}_{max})$ is a pendent vertex  adjacent to a vertex $v\in V(C^{\prime}T^{2}_{max})$ of degree 4. Without loss of generality, we assume that $z$ lies on the unique $u-w$ path. Let $w_1$ and $w_2$ be the neighbors of $w$ different from $z$. If $T^{\prime}=C^{\prime}T^{2}_{max}-\{w_{1}w,w_{2}w\}+\{uw_{1},uw_{2}\}$, then it can easily be observed that $T^{\prime}\in \mathcal{CT}^*_{n,b}$ and $
M_{2}(C^{\prime}T^{2}_{max})- M_{2}(T^{\prime})
=-2<0,
$
which is a contradiction to the choice of $C^{\prime}T^{2}_{max}$.
\end{proof}

\begin{lem}\label{lem-bm3}
If the tree $C^{\prime}T^{2}_{max}\in \mathcal{CT}^*_{n,b}$ contains a pendent vertex adjacent to a branching vertex, then it does not contain any pendent path of length greater than 2 where $1\leq b < \frac{n}{2}-1$.
\end{lem}

\begin{proof}
The proof is fully analogous to that of Lemma \ref{sm2}.
\end{proof}

\begin{lem}\label{lem-bm5}
For $1\leq b < \frac{n}{2}-1$\,, each vertex of degree 3 (if exists) of the tree $C^{\prime}T^{2}_{max}\in \mathcal{CT}^*_{n,b}$ has at most one neighbor of degree 4.
\end{lem}

\begin{proof}
Suppose, on the contrary, that $z\in V(C^{\prime}T^{2}_{max})$ is a vertex of degree 3 and that the vertices $x, y \in N(z)$ have degree 4. Then, by Lemma \ref{lem-bm4}, the tree $C^{\prime}T^{2}_{max}$ does not contain any vertex of degree 2. Let $u \in V(C^{\prime}T^{2}_{max})$ be a pendent vertex adjacent to a branching vertex $v\ne z$ (it is possible that the vertex $v$ is coincident with $x$ or $y$). If $T^{\prime}= C^{\prime}T^{2}_{max}-\{xz,zy,uv\}+\{xy,uz,zv\}$, then $T^{\prime}\in \mathcal{CT}^*_{n,b}$ and
$
M_{2}(C^{\prime}T^{2}_{max})- M_{2}(T^{\prime})= 5-2d_{v}<0,
$
which is a contradiction to the definition of $C^{\prime}T^{2}_{max}$.
\end{proof}

\begin{lem}\label{lem-bm6a}
For $1\leq b < \frac{n}{2}-1$, the tree $C^{\prime}T^{2}_{max}\in \mathcal{CT}^*_{n,b}$ has at least one vertex of degree 4 and the graph induced by the vertices of degree 4 of $C^{\prime}T^{2}_{max}$ is a tree.
\end{lem}

\begin{proof}
If $1\le b < \frac{n-2}{3}$ then by using Lemmas \ref{zb-A1} and \ref{zb-A2}, we have $n_3=0$ and the inequality $b\ge1$ implies that $n_4>0$. Hence, by Lemma \ref{lem-bm1}, the graph induced by the vertices of degree 4 of $C^{\prime}T^{2}_{max}$ is a tree.
In what follows, we assume that $\frac{n-2}{3}\le b < \frac{n}{2}-1$. By Lemmas \ref{zb-A1} and \ref{zb-A2}, it holds that $n_2=0$ and $n_4=n-2b-2>0$. By Lemma \ref{lem-bm1}, every internal path of  $C^{\prime}T^{2}_{max}$ has length 1. Suppose contrarily that the graph induced by the vertices of degree 4 of $C^{\prime}T^{2}_{max}$ is not a tree. Let $u_0u_1u_2\cdots u_r$ be a path of length at least 2 in $C^{\prime}T^{2}_{max}$ such that $d_{u_0}=d_{u_r}=4$ and $d_{u_1}=d_{u_2}=\cdots=d_{u_{r-1}}=3$. Let $v\in V(C^{\prime}T^{2}_{max})$ be a pendent vertex adjacent to a branching vertex $w$. If $T^{\prime}= C^{\prime}T^{2}_{max}-\{u_0u_1,u_{r-1}u_r,vw\}+\{u_0u_r,u_1v,u_{r-1}w\}$, then $T^{\prime}\in \mathcal{CT}^*_{n,b}$ and
$
M_{2}(C^{\prime}T^{2}_{max})- M_{2}(T^{\prime})= 5-2d_{w}<0,
$
which is a contradiction to the definition of $C^{\prime}T^{2}_{max}$. This completes the proof.
\end{proof}

Finally, we are now able to give the proof of Theorem \ref{thm-4}.

\begin{proof}[\bf Proof of Theorem \ref{thm-4}]
Recall that we have denoted by $C^{\prime}T^{2}_{max}$ the tree attaining the maximal second Zagreb index among all the members of $\mathcal{CT}^*_{n,b}$. Thus, $M_2(BT)\le M_2(C^{\prime}T^{2}_{max})$ with equality if and only if $BT \cong C^{\prime}T^{2}_{max}$. In what follows, we determine $M_2(C^{\prime}T^{2}_{max})$.

By Lemmas \ref{zb-A1} and \ref{zb-A2}, the degree sequence $DS(C^{\prime}T^{2}_{max})$ of $C^{\prime}T^{2}_{max}$ is
\[
DS(C^{\prime}T^{2}_{max})
=
\begin{cases}
(\underbrace{4,4,...,4}_{b},\underbrace{2,2,...,2}_{n-3b-2},\underbrace{1,1,...,1}_{2b+2}) &\text{if \  $1\le b<\frac{n-2}{3}$\,, }\\[6mm]
(\underbrace{4,4,...,4}_{n-2b-2},\underbrace{3,3,...,3}_{3b-n+2},\underbrace{1,1,...,1}_{n-b}) &\text{if \  $\frac{n-2}{3}\le b \le \frac{n}{2}-1$.}
\end{cases}
\]
Now, by Lemmas
\ref{lem-bm1} -- \ref{lem-bm6a} one can conclude that the tree $C^{\prime}T^{2}_{max}$ belongs to $\mathcal{BT'}_{1}(n,b)$ or $\mathcal{BT'}_{2}(n,b)$.\\

\textit{Case 1.} $1\le b < \frac{n-2}{3}$.\\
In this case, we have $n_1=2b+2$, $n_2=n-3b-2$, $n_3=0$, $n_4=b$  and hence (by Lemma \ref{lem-bm6a}), it holds that
\begin{equation}\label{Eq-Ba0}
x_{4,4}=n_4-1=b-1.
\end{equation}

\textit{Subcase 1.1} $1\le b \le \frac{n-4}{5}$.\\
In this subcase, it holds that $x_{1,4}=0$ and hence from \eqref{Eq557b} and \eqref{Eq-Ba0}, it follows that $x_{1,2}=x_{2,4}= 2b+2$ and $x_{2,2}=n-5b-4$. Thus,
$$M_2(C^{\prime}T^{2}_{max})=4n+16b-12.$$

\textit{Subcase 1.2} $\frac{n-4}{5} < b < \frac{n-2}{3}$.\\
In this subcase, we have $x_{1,4}\ne0$. Thus, it holds that $x_{2,2}=0$ (by Lemmas \ref{lem-bm1} and \ref{lem-bm3}) and hence from \eqref{Eq557b} and \eqref{Eq-Ba0}, it follows that $x_{1,2}=x_{2,4}=n-3b-2$,  $x_{1,4}=5b-n+4$. Thereby,
$$M_2(C^{\prime}T^{2}_{max})=6n+6b-20.$$

\textit{Case 2.} $\frac{n-2}{3}\le b < \frac{n}{2}-1$.\\
In this case, it holds that $n_1=n-b$, $n_2=0$, $n_3=3b-n+2$, $n_4=n-2b-2>0$ and hence (by Lemma \ref{lem-bm6a}), it holds that
\begin{equation}\label{Eq-Ba1}
x_{4,4}=n_4-1=n-2b-3.
\end{equation}

\textit{Subcase 2.1}  $\frac{n-2}{3}\le b< \frac{3n-4}{7}$.\\
In this subcase, we have $x_{1,4}\ne0$, which forces that $x_{3,3}=0$ (by Lemma \ref{lem-bm2}) and hence from \eqref{Eq557b} and \eqref{Eq-Ba1}, we get  $x_{1,4}=3n-7b-4$, $x_{1,3}=6b-2n+4$, $x_{3,4}=3b-n+2$. Thus,
$$M_2(C^{\prime}T^{2}_{max})=10n-6b-28.$$

\textit{Subcase 2.2} $\frac{3n-4}{7} \le b < \frac{n}{2}-1$.
We note that $x_{1,4}=0$ in this subcase and thereby from \eqref{Eq557b} and \eqref{Eq-Ba1}, it follows that $x_{1,3}=n-b$, $x_{3,4}=2n-4b-2$, $x_{3,3}=7b-3n+4$.
Hence,
$$M_2(C^{\prime}T^{2}_{max})=16n-20b-36.$$
This completes the proof.
\end{proof}

\end{document}